\documentclass[12pt,reqno]{amsart}
\usepackage{latexsym, amssymb, amscd, mathrsfs,geometry, amsmath, rotating,amsthm,color}

\newtheorem{thm}{Theorem}[section]
\newtheorem{lem}[thm]{Lemma}
\newtheorem{prop}[thm]{Proposition}
\newtheorem{cor}[thm]{Corollary}
\newtheorem{Def}{Definition}[section]
\newtheorem{ex}{Example}[section]

\newtheorem{rem}{Remark}

\def\a{\alpha}
\def\b{\beta}

\def\la{\lambda}
\def\Fg{\mathfrak{g}}
\def\Fgl{\mathfrak{gl}}
\def\Fsl{\mathfrak{sl}}
\def\Fp{\mathfrak{p}}
\def\Fh{\mathfrak{h}}
\def\CC{\mathbb{C}}
\def\ZZ{\mathbb{Z}}
\def\U{\mathcal{U}}
\def\Ind{\mathrm{Ind}}
\def\Supp{\mathrm{Supp}}
\def\deg{\mathrm{deg}}

\def\h{\mathrm{ht}}
\def\Span{\mathrm{Span}}

\def\rank{\mathrm{rank}}

\def\W{\mathcal{W}}
\def\Ba{\boldsymbol{\mathrm{a}}}
\def\beqn{\begin{equation*}}
	\def\eeqn{\end{equation*}}
\def\beq{\begin{equation}}
	\def\eeq{\end{equation}}
\def\({\left(}
\def\){\right)}
\newcounter{Par}[subsection]

\counterwithin{equation}{section}

\title{Quasi-Whittaker modules}
\author{Cunguang Cheng, Wenting Gao, Shiyuan Liu,  Kaiming Zhao, Yueqiang Zhao}

\date{}
\begin{document}
	\maketitle
	\begin{abstract}
In this paper, a general setting is proposed to define  a class of modules over  nonsemisimple Lie algebras $\mathfrak{g}$ induced by a nonperfect ideal $\mathfrak{p}$. This class of Lie algebras includes   many well-known Lie algebras, and  some of this class  of modules are Whittaker modules and others are not. We call these modules quasi-Whittaker modules.
By introducing a new concept: the Whittaker annihilator for   universal quasi-Whittaker modules, we are able to determine  the necessary and sufficient conditions for the irreducibility of the universal quasi-Whittaker modules. In the reducible case, we can obtain some maximal submodules. In particular, we classify the irreducible quasi-Whittaker modules for many Lie algebras, and obtain a lot of irreducible  smooth $\W_n^+$-modules of height $2$. 
	
\end{abstract}


2020 MCS: 17B10, 17B20, 17B65, 
	
\section{Introduction}
Whittaker modules are a class of  important   modules in representation theory that were first defined for $\mathfrak{sl}_{2}$ by Arnal and Pinzcon \cite{AG}. In \cite{K}, B. Kostant formally  introduced Whittaker modules for finite-dimensional complex semisimple Lie algebras,  which can be generalized to any Lie algebras with a triangular decomposition. Let $\mathfrak{g}=\mathfrak{n}_{-}\oplus\mathfrak{h}\oplus\mathfrak{n}_{+}$ be a complex semisimple Lie algebra and $\eta:\mathfrak{n}_{+}\to\mathbb{C}$ a Lie algebra homomorphism. A $\mathfrak{g}$-module is a Whittaker module if $x - \eta(x)$ acts locally nilpotently on it for any $x\in\mathfrak{n}_{+}$. Various classes of simple Whittaker modules in the above sense were studied and classified by many authors (see \cite{B,K,M1,M2,MS}).

Analogues of Whittaker modules can be defined and studied in various similar scenarios, as seen in the following   papers. Whittaker modules (or generalized Whittaker modules) have been explored over the Virasoro algebra and its various generalizations (see \cite{OW,LGZ,LPX, LZ}), the affine Kac-Moody algebras (see \cite{ALZ, CGLW}), the Lie algebra of polynomial vector fields (see \cite{ZL}),  the classical Lie superalgebras (see \cite{C}), and even for quantum groups like $\mathcal{U}_q(\mathfrak{sl}_n)$ (see \cite{O,XGZ, XZ}).  

A general setup of \emph{Whittaker pair}   which one can define  Whittaker modules was proposed in \cite{BM}, see also \cite{MZ}. A Whittaker pair $(\mathfrak{g}, \mathfrak{n})$ consists of a Lie algebra $\mathfrak{g}$ and a locally nilpotent subalgebra $\mathfrak{n}$ of $\mathfrak{g}$ such that the action of $\mathfrak{n}$ on the adjoint $\mathfrak{n}$-module $\mathfrak{g}/\mathfrak{n}$ is locally nilpotent.
In \cite{CCS}, Y. Cai et al, studied a type of Whittaker modules over the Schrodinger algebra induced by the Heisenberg subalgebra. To distinguish from the case of the Borel subalgebra and parabolic subalgebra, they named these modules quasi-Whittaker modules.  Subsequently, quasi-Whittaker modules over the Euclidean algebra $\mathfrak{e}(3)$, the conformal Galilei algebras, the $(2 + 1)$- dimensional spacetime Schrödinger algebra, and the $n$-th Schrödinger algebra have all been studied (see \cite{CSZ,CC,CL,JLZ,CW}). In particular, the above-mentioned  works exhibit a certain degree of uniformity. Inspired by this, we will uniformly study a class of modules over  nonsemisimple Lie algebras $\mathfrak{g}$ induced by a nonperfect ideal $\mathfrak{p}$. We still call them quasi-Whittaker modules. \ {Generally speaking, if $\mathfrak{p}$ is not locally nilpotent, the pair $(\mathfrak{g}, \mathfrak{p})$ does not form a \emph{Whittaker pair}.} The purpose of this paper is  to set up a uniform way to study this class of modules for many Lie algebras, including finding quasi-Whittaker vectors, studying the irreducibility of universal quasi-Whittaker modules, and obtaining a specific class of irreducible quasi-Whittaker modules. 

The paper is arranged as follows. In Section 2,  we define quasi-Whittaker modules over nonssemisimple Lie algebras with a nonperfect ideal,
and establish some elementary results associated to quasi-Whittaker modules for later use. In Section 3,  we introduce the new concept: the Whittaker annihilator for   universal quasi-Whittaker modules,    determine all quasi-Whittaker vectors (see Theorem \ref{T3.1}) and the irreducibilities of the universal quasi-Whittaker modules (see Theorem \ref{T3.5}). For the reducible cases, we determine some irreducible quotients (see Lemma \ref{L3.4} and Theorem \ref{T3.11}). Moreover, a sufficient condition is given for classifying irreducible quasi-Whittaker modules (see Theorem \ref{T3.10}). In Section 4, we give some examples and applications, to recover many known results and to provide many new irreducible modules. In particular, we obtain a lot of irreducible  smooth $\W_n^+$-modules of height $2$. 

Throughout this paper, we use 
$\mathbb{Z}$ to denote the set of integers, $\mathbb{N}$
  for positive integers,
$\mathbb{Z}_{+}$
  for nonnegative integers, and 
$\mathbb{C}$ for complex numbers.  All vector spaces are over $\mathbb C$.
\section{Quasi-Whittaker modules}

  For a totally ordered set  $I$ (finite, countable or uncountable), denote  by
\beqn
\ZZ_{+}^{I}=\{\a: I\rightarrow \ZZ_+\mid \a(i)\neq0~\text{only for finitely many  }i\in I\}.
\eeqn
For $\a\in\ZZ_+^{I}$, we write $\a_i=\a(i)$ for each $i\in I$. Define the {\bf size} of  $\a$ to be
\beqn
|\a|=\sum_{i\in I}\a_i.
\eeqn
For $i\in I$, denote by $\varepsilon_i\in\ZZ^I_+$  the element such that
\beqn
\varepsilon_{ij}=\delta_{ij},~\forall j\in I.
\eeqn
For $\a\in\ZZ_+^I$ we use  the notation
\beqn
	x^{\a}=x_{i_1}^{\a_{i_1}}x_{i_2}^{\a_{i_2}}\cdots x_{i_n}^{\a_{i_n}},
\eeqn
where $i_1<i_2<\cdots<i_n$ are chosen such that 
\beqn
\a_k=0,~\forall k\in I\backslash\{i_1,i_2,\cdots,i_n\}.
\eeqn
Here, $x^0$ is viewed as $1$. 

Define a total order ~``$>$" on $\ZZ_+^I$ by
\begin{equation*}\label{xu}
	\a>\b\Longleftrightarrow |\a|>|\b|,~\text{or } |\a|=|\b|~ \text{and }\exists  i~\text{such that}~\a_i>\b_i \text{ and } \a_j=\b_j\text{ for } j<i.
\end{equation*}

Let $\a\in\ZZ_+^I$ with $|\a|>0$.  Denote by $\h(\a)$ the largest index $k\in I$  such that $\a_k\neq 0$, called the {\bf height} of $\a$. We write
\beqn
\hat{\a}=\a-\varepsilon_{\h(\a)}.
\eeqn
The following fact is easy to verify.
\begin{lem}\label{NL2.1}
	Let $\a,\b\in\ZZ^I_+$  with positive size such that  $\a>\b$. Then $\hat{\a}\ge\hat{\b}$. Moreover, the equality holds if and only if  there exists some $j>\h(\a)$ such that $\b=\hat{\a}+\varepsilon_j$.
\end{lem}

In the following, let $\Fg$ be nonsemisimple Lie algebra and  $\mathfrak{p}$ a nonperfect ideal  of $\Fg$. 
\begin{Def}
	Let $\phi:\Fp\rightarrow\CC$ be any Lie algebra homomorphism and $V$  a $\Fg$-module. 
	\begin{itemize}
		\item[(1)] A  vector $v\in V$  is called a \textbf{quasi-Whittaker vector of type $\phi$} if $pv=\phi(p)v$ for each $p\in\Fp$.
		
		\item[(2)] $V$ is said to be a \textbf{quasi-Whittaker module of type $\phi$} if $V$ is generated by a quasi-Whittaker vector of type $\phi$.
	\end{itemize}
\end{Def}

Note that $\phi=0$ if  $\mathfrak{p}$ is a perfect ideal  of $\Fg$. So we assume that  $\mathfrak{p}$ is a nonperfect ideal  of $\Fg$ and $\phi\ne0$.
Let $\phi:\Fp\rightarrow\CC$ be a given nonzero Lie algebra homomorphism. We define a one dimensional $\Fp$-module $\CC w_{\phi}$ by 
\beq\label{F2.1}
pw_{\phi}=\phi(p)w_{\phi},~\forall p\in\Fp.
\eeq
Denote by
\beqn
W(\phi)=\Ind_{\Fp}^{\Fg}\CC w_{\phi}
\eeqn
the corresponding induced module. Then $W(\phi)$ is a quasi-Whittaker module of type $\phi$ with a cyclic quasi-Whittaker vector $1\otimes w_{\phi}$. We write  $w_{\phi}=1\otimes w_{\phi}$ for convenience. 
For any quasi-Whittaker module $V$ generated by a quasi-Whittaker vector $v$ of type $\phi$, there exists a surjective $\Fg$-module homomorphism $\Phi:W(\phi)\rightarrow V$ such that $\Phi(w_{\phi})=v$.  We call $W(\phi)$ \textbf{the universal quasi-Whittaker module of type $\phi$}.

\begin{Def} We define the {\bf Whittaker annihilator} of $\phi$ as 
\beq\label{F2.2}
\Fg^{\phi}=\{g\in\Fg\mid \phi([g,p])=0,~\forall p\in\Fp\}.
\eeq
\end{Def}

\begin{lem}\label{L2.1}
	$\Fg^{\phi}$ is a subalgebra of $\Fg$ containing $\Fp$. 
\end{lem}

\begin{proof}
	It is clear that $0\in\Fg^{\phi}$.  For  $g_1,g_2\in\Fg^{\phi}$ and $\la_1,\la_2\in\CC$, we have
	\begin{align*}
		&\phi([\la_1g_1+\la_2g_2,p])=\la_1\phi([g_1,p])+\la_2\phi([g_2,p])=0+0=0,\\
		&\phi([[g_1,g_2],p])=\phi([[g_1,p],g_2])+\phi([g_1,[g_2,p]])=0+0=0
	\end{align*}
	for each $p\in\Fp$.  It follows that 
	$\la_1g_1+\la_2g_2,[g_1,g_2]\in \Fg^{\phi}.$ 
	Hence, $\Fg^\phi$ is a subalgebra of $\Fg$.  Since $\phi$ is a Lie algebra homomorphism, one has $\phi([\Fp,\Fp])=0$, which implies that $\Fp\subseteq\Fg^{\phi}$.
\end{proof}

 In the following discussion, we choose a basis of $\Fg$ in the following form
\beq\label{F2.3}
\underbrace{x_i,~i\in I;~ \underbrace{y_j,~j\in J;~\underbrace{p_k,~k\in K}_{\text{a basis of }\Fp}}_{\text{a basis of }\Fg^{\phi}}}_{\text{a basis of} ~\Fg},
\eeq
where $I,J$ and $K$ are totally ordered sets (or empty). We denote by $\U(\Fg),\U(\Fg^{\phi})$ and $\U(\Fp)$ the universal  enveloping algebras of $\Fg,\Fg^{\phi}$ and $\Fp$ respectively. Then there exist natural inclusions
\beq\label{F2.4}
\U(\Fp)\subseteq \U(\Fg^{\phi})\subseteq \U(\Fg).
\eeq
Moreover, by PBW theorem, $\U(\Fg)$ is a free right $\U(\Fg^{\phi})$-module with basis
\beqn
x^{\a},~\a\in\ZZ_+^I,
\eeqn
and is a free right $\U(\Fp)$-module with basis
\beqn
x^\a y^{\b},~\a\in\ZZ_+^I,~\b\in\ZZ_+^J.
\eeqn
 For a $\Fg$-module  $V$, we write 
\beqn
V_{\phi}=\{v\in V\mid pv=\phi(p)v,~\forall p\in\Fp\}.
\eeqn
Then $V_{\phi}$ is a subspace of $V$. Moreover, we have

\begin{lem}\label{L2.2} For a $\Fg$-module  $V$, the subspace 
	$V_{\phi}$ is $\Fg^\phi$-invariant. 
\end{lem}

\begin{proof}
	Let $v\in V_\phi$ and $y\in\Fg^{\phi}$. Then we have
	\beqn
	pyv=[p,y]v+ypv=\phi([p,y])v+y\(\phi(p)v\)=0+\phi(p)yv,~\forall p\in\Fp.
	\eeqn
	It follows that $yv\in V_{\phi}$.
\end{proof}

For $v\in V_{\phi}$ and $u\in \mathcal{U}(\Fg)$, it is easy to check that
\beq\label{F2.5}
(p-\phi(p))uv=[p,u]v,~\forall p\in\Fp.
\eeq

\begin{lem}\label{L2.3}
	Assume that $|I|>0$. Let $\a\in\ZZ_+^I$  with height $k$. For a $\Fg$-module  $V$ we have
	\begin{equation*}\label{N2-F4}
		\(p-\phi(p)\)x^\a v=\a_k\phi([p,x_k])x^{\hat{\a}}v+\text{a linear combination of } x^\beta v\text{'s}~\text{with~} \beta<\hat{\a}
	\end{equation*}
	for each $p\in\Fp$ and $v\in V_{\phi}$.
\end{lem}

\begin{proof}
	Assume that $	x^{\a}=x_{i_1}^{\a_{i_1}}x_{i_2}^{\a_{i_2}}\cdots x_{i_k}^{\a_{i_k}}$. Then 
	\begin{equation}\label{2.6}
			(p-\phi(p))x^\a v=[p,x^\a]v=\sum_{j=1}^kx_{i_1}^{\a_{i_1}}x_{i_2}^{\a_{i_2}}\cdots [p,x_{i_j}^{\a_{i_j}}]\cdots x_{i_k}^{\a_{i_k}}
	\end{equation}

$$
=\a_k\phi([p,x_k])x^{\hat{\a}}v+\text{a linear combination of } x^\beta v\text{'s}~\text{with~} \beta<\hat{\a}.$$
\end{proof}

\begin{thm}
	Let $V$ be an irreducible $\mathfrak{g}$-module. Assume that $\mathfrak{p}$ is a   solvable ideal of $\mathfrak{g}$. Then the following two conditions are equivalent:
	\begin{enumerate}
		\item[(1)] $V$ is a locally finite $\mathfrak{p}$-module;
		\item[(2)] $V$ is a quasi-Whittaker module for $\mathfrak{g}$.
	\end{enumerate}
\end{thm}
\begin{proof}
	$(1)\Rightarrow(2)$. Assume that $V$ is a locally finite $\mathfrak{p}$-module. Taking any nonzero $v_0\in V$, the  $\mathfrak{p}$-submodule $V_1\subseteq V$   generated by  $v_0$ is  finite-dimensional. By Lie's theorem, there exists   a common eigenvector  $w$  for all elements of $\mathfrak{p}$.  Hence, there exists a linear map $\phi:\Fp\rightarrow \CC$ such that
	\beqn
	pw=\phi(p)w,~\forall p\in\Fp.
	\eeqn
	Since $V$ is irreducible as a $\mathfrak{g}$-module, we see that $V = U(\mathfrak{g})w$. Clearly, $V$ is a quasi-Whittaker module of type $\phi$.
	
	$(2)\Rightarrow(1)$. Let $V$ be an irreducible quasi-Whittaker module generated by a  quasi-Whittaker vector $v_{\phi}$ of type $\phi$.  
	  Then any $v\in V$ can be   written as  
	\[
	v = \sum_{\substack{\alpha \in \mathbb{Z}_+^I, \beta \in \mathbb{Z}_+^J}} a_{\alpha,\beta} x^{\alpha} y^{\beta} v_{\phi}
	\]  
	with finitely many nonzero $a_{\alpha,\beta}$'s.  
	From \eqref{2.6}, one can easily decuce that
	\beqn
	\mathcal{U}(\Fp) x^{\alpha} y^{\beta} v_{\phi} \subseteq \text{Span}_{\mathbb{C}}\{ x^{\alpha'} y^{\beta} v_{\phi} \mid \alpha'\in \ZZ^I~\text{with}~\alpha'_i \leq \alpha_i,~\forall i \in I \}
	\eeqn 
	for each $\a\in\ZZ_+^I$ and $\b\in\ZZ_+^I$. In particular, 
	\beqn
	\dim  \mathcal{U}(\mathfrak{p}) x^{\alpha} y^{\beta} v_{\phi} < \infty.
	\eeqn  
	It follows that  $\dim \U(\Fp)v<\infty$  since
	\beqn
	\U(\Fp)v\subseteq \sum_{a_{\a,\b}\neq 0} \U(\Fp)x^{\a}y^{\b}v_{\phi}.
	\eeqn
	Hence, $V$ is a locally finite  $\Fp$-module.
\end{proof}
 
We also use the symbol ``$>$" to denote the total order on $\ZZ_+^I\times\ZZ_+^J$ defined by 
\begin{align*}
(\a,\b)>(\a',\b')\Longleftrightarrow&  |\a|+|\b|>|\a'|+|\b'|;~\text{or~}  |\a|+|\b|=|\a'|+|\b'|~\text{but~} \a>\a';\\
&\text{or}~|\a|+|\b|=|\a'|+|\b'|~\text{and}~\a=\a'~\text{but~} \b>\b'.
\end{align*}

	\begin{prop}\label{P2.4}
	Each nonzero quasi-Whittaker vector in $W(\phi)$ is of type $\phi$.
\end{prop}

\begin{proof}
	Let $w=uw_{\phi}\in W({\phi})$ be a nonzero quasi-Whittaker vector of type $\varphi$.  Write  
	\begin{equation*}
		w=a x^{\a}y^{\beta}w_{\phi}+\text{a linear combination of } x^{\a'} y^{\beta'}w_{\phi}\text{'s}~\text{with~}(\a',\beta')<(\a,\b) ,
	\end{equation*}
	where $a \neq 0$. 
	For each $p\in\Fp$, one has
	\begin{equation*}
		pw=\varphi(p)w=a \varphi(p)x^{\a}y^{\beta}w_{\phi}+\text{a linear combination of } x^{\a'} y^{\beta'}w_{\phi}\text{'s}\\
		~\text{with~}(\a',\beta')<(\a,\b)
	\end{equation*}
	and
	\begin{align*}
		pw=\([p,u]+up\)w_{\phi}=a \phi(p)x^{\a}y^{\beta}w_{\phi}+&\text{a linear combination of } x^{\a'} y^{\beta'}w_{\phi}\text{'s}\\
		&~~~~~~~~~~~~~~~~~~~\text{with~}(\a',\beta')<(\a,\b).
	\end{align*}
	Comparing the above two equations, we can obtain that $\varphi(p)=\phi(p)$. It follows that $\varphi=\phi$.
\end{proof}

 \section{Irreducibility of quasi-Whittaker modules}
 
In this section, we first characterize all quasi-Whittaker vectors, then determine the irreducibility of the universal quasi-Whittaker modules. We also classify the irreducible quasi-Whittaker modules with a unique nonzero quasi-Whittaker vector up to a scalar.
 
 Throughout this section, we let $\Fg$ be nonsemisimple Lie algebra,  $\mathfrak{p}$ a nonperfect ideal  of $\Fg$, 
and $\phi:\Fp\rightarrow\CC$ a Lie algebra homomorphism.

 \subsection{Irreducibility of $W(\phi)$}
 
 We now determine all the quasi-Whittaker vectors  in $W(\phi)$ defined in Section 2.  
We denote by 
\beqn
\CC[y]=\Span_{\CC}\{y^{\beta}\mid \beta\in\ZZ_+^J\}\subseteq \mathcal{U}(\Fg).
\eeqn
Then $\CC[y]$ has a basis
\beqn
y^{\beta},~\beta\in\ZZ_+^J.
\eeqn
We see that
\beq\label{F3.1}
\U(\Fg^{\phi})w_{\phi}=\CC[y]w_{\phi}.
\eeq
Let $\CC[t]=\CC[t_j\mid j\in J]$ denote the polynomial algebra in variables $t_j,~j\in J$ over $\CC$.  We understand $\CC[y]=\CC[t]=\CC$ when $J=\emptyset$. Then there exists a vector space  isomorphism 
\beq\label{F3.2}
\hat{ }: \CC[y]\longrightarrow \CC[t],~y^\beta\longmapsto t^\beta,~\forall \beta\in\ZZ_{+}^J.
\eeq

Note that $W(\phi)$ has a basis
\beqn
x^{\a}y^{\beta} w_{\phi},~\a\in\ZZ_+^{I},~\beta\in\ZZ_+^{J}.
\eeqn
Hence, any nonzero $v$ in $W(\phi)$ can be uniquely written as 
\beq\label{F3.3}
v=\sum_{\a\in\ZZ_+^I}x^\a f_{\a}w_{\phi}
\eeq
with finitely many nonzero $f_\a\in\CC[y]$. 
Assume that $|I|>0$. Define the  {\bf support} of $v$ to be 
\begin{equation*}
	\Supp(v)=\{ \a\in\ZZ_+^I\mid f_{\a}\neq 0\}.
\end{equation*}
We define the 
 {\bf degree} of $v$ as
\begin{equation*}
	\deg(v)=\max\Supp(v)
\end{equation*}
with respect to the order ~``$>$" on $\ZZ_+^I$.  

\begin{thm}\label{T3.1} For  the universal quasi-Whittaker module $W(\phi)$, we have 
 $W(\phi)_{\phi}=\CC[y]w_{\phi}$.
\end{thm}

\begin{proof} We can easily verify that 
	$	\CC[y]w_{\phi} \subseteq W(\phi)_{\phi}.$
	It is enough to show that  $W(\phi)_{\phi}\subseteq \CC[y]w_{\phi}$. To the contrary, suppose that $W(\phi)_{\phi}\not\subseteq \CC[y]w_{\phi}$. Then there exists some nonzero $v\in W(\phi)_{\phi}$ with $\deg(v)=\a>0$. Let $k=\h(\a)$ and 
	 \beqn
	 \ell=\max\{\h(\gamma)\mid \gamma\in\Supp(v)\}.
	 \eeqn
	 Then $k\le\ell$ and the set 
	 \beqn
	 T=\{j\in I\mid k\le j\le \ell~\text{and~}\gamma_j\neq 0~\text{for some }\gamma\in\Supp(v) \}
	 \eeqn
	 is finite. 
	 Assume that 
	 \beqn
	 T=\{k=i_1<i_2<\cdots<i_m=\ell\}.
	 \eeqn
	 Denote by
	 \beqn
	 \a^{(t)}=\a-\varepsilon_k+\varepsilon_{i_t},~t=1,2,\cdots,m.
	 \eeqn
	Then 
	\beqn
	\a=\a^{(1)}>\a^{(2)}>\a^{(3)}>\cdots>\a^{(m)}
	\eeqn
	are adjacent elements in the set
	\beqn
	\Supp(v)\cup\{\a^{(1)},\a^{(2)},\a^{(3)},\cdots,\a^{(m)}\}.
	\eeqn
 with respect to the order ``$>$". Hence, $v$ is expressed as
		\begin{equation*}
		v=x^{\a}f_\a w_{\phi}+\sum_{t=2}^mx^{\a^{(t)}}f_{\a^{(t)}}w_{\phi}+\text{lower degree terms},
	\end{equation*}
with $0\neq f_\a\in \CC[y]$. Note that $\hat{f}_{\a}\neq 0$ since $f_{\a}\neq 0$. We can find some 
$\Ba=(a_j)_{j\in J}$ with $a_j\in\CC$ 
such that  $\hat{f}_\a(\Ba)\neq 0$. Then
\beqn
\a_k\hat{f}_\a(\Ba)x_k+\hat{f}_{\a^{(2)}}(\Ba)x_{i
_2}+\cdots+\hat{f}_{\a^{(m)}}(\Ba)x_{i_m}\notin \Fg^{\phi},
\eeqn
Hence, by the definition of $\Fg^{\phi}$ there exists $p\in\Fp$ such that
\beqn
\phi\([\a_k\hat{f}_\a(\Ba)x_k+\hat{f}_{\a^{(2)}}(\Ba)x_{i_2}+\cdots+\hat{f}_{\a^{(m)}}(\Ba)x_{i_m},p]\)\neq0,
\eeqn
i.e.,
\beqn
\a_{k}\phi([x_k,p])\hat{f}_\a(\Ba)+\phi([x_{i_2},p])\hat{f}_{\a^{(2)}}(\Ba)+\cdots \phi([x_{i_m},p])\hat{f}_{\a^{(m)}}(\Ba)\neq0,
\eeqn
which implies that
\beqn
\a_{k}\phi([x_k,p])\hat{f}_\a+\phi([x_{i_2},p])\hat{f}_{\a^{(2)}}+\cdots \phi([x_{i_m},p])\hat{f}_{\a^{(m)}}\neq0.
\eeqn
It follows that 
\beqn
\a_{k}\phi([x_k,p])f_\a+\phi([x_{i_2},p])f_{\a^{(2)}}+\cdots \phi([x_{i_m},p])f_{\a^{(m)}}\neq0.
\eeqn
By Lemma \ref{NL2.1} and Lemma \ref{L2.3}, we have
	\begin{align*}
	\(p-\phi(p)\)v=&-x^{\hat{\a}}\(	\a_k\phi([x_k,p])f_{\a}+\phi\([x_{i_2},p]\)f_{\a^{(2)}}+\cdots+\phi\([x_{i_m},p]\)f_{\a^{(m)}}\) w_{\phi}\\
	&+\text{lower degree terms}\neq 0,
	\end{align*}
which contradicts $v\in W(\phi)_{\phi}$. The proof is completed.
\end{proof}

From the above proof process, one can deduce the following proposition easily 
\begin{prop}\label{P3.2}
	Any nonzero $\Fp$-invariant subspace of $W(\phi)$ contains a nonzero quasi-Whittaker vector of type $\phi$.
\end{prop}
\begin{proof}
 Let $ M $ be a non-zero $ \mathfrak{p} $-invariant subspace of $ W(\phi) $. Take $ v $ to be a non-zero element in $ M $ with $ |\text{deg}(v)| $ minimized. We claim that $ v \in W(\phi)_{\phi} $. Otherwise, assume $ \deg(v) = \alpha $, then necessarily $ |\alpha| > 0 $. We may assume that $v$ is of the form in (\ref{F3.3}). From the proof of Theorem 3.1, we have some $ p \in \mathfrak{p} $ such that 
 
 \[
 \begin{aligned}
 	(p - \phi(p))v = {}& -x^{\hat{\alpha}} \left( \alpha_k \phi([x_k, p]) f_{\alpha} + \phi([x_{i_2}, p]) f_{\alpha^{(2)}} + \cdots + \phi([x_{i_m}, p]) f_{\alpha^{(m)}} \right) w_{\phi} \\
 	&+ \text{lower degree terms} \neq 0
 \end{aligned}
 \]
 
 Since $ |\hat{\alpha}| = |\alpha| - 1 $, $ (p - \phi(p))v $ is a non-zero vector in $ M $ whose size of support  is smaller than that of $ v $, which contradicts the choice of $ v $.
\end{proof}

Now, we can determine the irreduciblity of $W(\phi)$. 

		  \begin{thm}\label{T3.5}  Let $\Fg$ be nonsemisimple Lie algebra,  $\mathfrak{p}$ a nonperfect ideal  of $\Fg$, 
and $\phi:\Fp\rightarrow\CC$ a Lie algebra homomorphism. Then
		$W(\phi)$ is irreducible if and only if $\Fg^{\phi}=\Fp$.
	\end{thm}

\begin{proof}
	Assume that $W(\phi)$ is irreducible. Suppose that $\Fg^{\phi}\neq\Fp$, i.e.,  $|J|>0$. Choose an element $j\in J$.    It is easy to see that $U(\Fg)y_jw_{\phi}$ does not contain $w_{\phi}$, which implies that $U(\Fg)y_jw_{\phi}$ is a proper submodule, a contradiction.
	
	Next, assume that $\Fg^{\phi}=\Fp$. By Theorem \ref{T3.1}, we have
	$W(\phi)_{\phi}=\CC w_{\phi}.$ Let $M$ be a nonzero proper $\Fg$-submodule of $W(\phi)$. Then $M$ is a $\Fp$-invariant subspace of $W(\phi)$. By Proposition \ref{P3.2}, $M$ contains a nonzero quasi-Whittaker vector. We 
	can deduce that $w_\phi\in M$. It follows that $M=W(\phi)$. Hence, $W(\phi)$ is an irreducible $\Fg$-module.
\end{proof}

As a corollary, we have

\begin{cor}\label{C3.6}
	Assume that $\Fg$ is finite-dimensional of the form $\Fg=\Fg_0\ltimes \Fp$. Let $g_1,g_2,\cdots,g_t$ be a basis of $\Fg_0$ and $p_1,p_2,\cdots,p_{\ell}$ a basis of $\Fp$. Denote by
	\begin{equation*}
		A_{\phi}=\begin{pmatrix}
			\phi([g_1,p_1]) & \phi([g_2,p_1]) & \cdots & \phi([g_t,p_1])\\
			\phi([g_1,p_2]) & \phi([g_2,p_2]) & \cdots & \phi([g_t,p_2])\\
			\vdots & \vdots &   & \vdots\\
			\phi([g_1,p_\ell]) & \phi([g_2,p_\ell]) & \cdots & \phi([g_t,p_\ell])
		\end{pmatrix}.
	\end{equation*}
	Then $W(\phi)$ is irreducible if and only if $\rank A_{\phi}=t.$ 
\end{cor}

\begin{proof}
	Just note that we have a decomposition
	\beq\label{F3.6}
	\Fg^{\phi}=\(\Fg_0\cap\Fg^{\phi}\)\ltimes \Fp,
	\eeq
	and the linear space $\Fg^{\phi}\cap \Fg_0$ is isomorphic to the  solution space of the  homogeneous linear system of equations with coefficient matrix $A_{\phi}$.  Hence, $W(\phi)$ is irreducible if and only if $\Fg^{\phi}\cap\Fg_0=0$, if and only if $\rank A_{\phi}=t.$ 
\end{proof}

 \subsection{Bland   quasi-Whittaker modules}
 
 In the subsection, we classify all   irreducible quasi-Whittaker modules with unique quasi-Whittaker vector (up to a scalar).
 Let $\Fg$ be nonsemisimple Lie algebra,  $\mathfrak{p}$ a nonperfect ideal  of $\Fg$, 
and $\phi:\Fp\rightarrow\CC$ a Lie algebra homomorphism.

 \begin{Def}
 Let $V$ be an irreducible quasi-Whittaker module over $\Fg$ of type $\phi$. We call $V$ \textbf{bland} if the quasi-Whittaker vector in $V$ is unique up to a scalar, i.e., $\dim V_{\phi}=1$.
 \end{Def}
 
Next Example \ref{Examaple3.1} provides irredcible quasi-Whittaker modules whose quasi-Whittaker vectors are not unique. 
 Now we also make the following definition. 
 
 \begin{Def}
 	A Lie algebra homomorphism  $\phi: \Fp\rightarrow \CC$ is said to be \textbf{extendable} if it can be extended to  a Lie algebra homomorphism    $\phi':\Fg^{\phi}\to \CC$.
 \end{Def}
 
 \begin{prop}
 Let $\phi:\Fp\rightarrow \CC$ be a Lie algebra homomorphism. Then $\phi$ is extendable if and only if  
 \beq
 [\Fg^{\phi},\Fg^{\phi}]\cap \Fp\subseteq \ker\phi.
 \eeq
 \end{prop}
 
 \begin{proof}
 	The necessity is obvious. We only prove the sufficiency. Assume that $[\Fg^{\phi},\Fg^{\phi}]\cap \Fp\subseteq \ker\phi.$ It is easy to check that   $\phi$ can be extended to a linear function $\phi'$ on $[\Fg^{\phi},\Fg^{\phi}]+\Fp$ by defining
 	$\phi'([\Fg^{\phi},\Fg^{\phi}])=0.$
 	Choose $y_s\in \Fg^\phi$  ($s\in S$) linearly independent  such that  
 	\beqn
 	\Fg^{\phi}=\Span_{\CC}\{y_s\mid s\in S\}\oplus \([\Fg^{\phi},\Fg^{\phi}]+\Fp\)
 	\eeqn
 	as a linear space. We extend $\phi'$ onto $\Fg^{\phi}$ by putting 
 	\beqn
 	\phi'(y_s)=\xi_s,~\forall s\in S
 	\eeqn
 	for arbitary $\xi_s\in\CC$. Then we obtain a Lie algebra homomorphism $\phi':\Fg^{\phi}\rightarrow \CC$. It follows that $\phi$ is extendable.
 \end{proof}
 
 \begin{lem}\label{L3.6}
 	If $\Fg$ admits  bland irreducible quasi-Whittaker modules of type $\phi$, then $\phi$ is extendable.
 \end{lem}
 
 \begin{proof}
 	Let $V$ be a bland irreducible quasi-Whittaker module  generated by a cyclic quasi-Whittaker vector $v$ of type $\phi$. Then $V_{\phi}=\CC v$.  By  Lemma \ref{L2.2},
 	\beqn
 	\CC v\subseteq U(\Fg^{\phi})v\subseteq V_{\phi}.
 	\eeqn 
 	It follows that $V_{\phi}=U(\Fg^{\phi})v$ is a one dimensional $\Fg^{\phi}$-module.  Hence, there exists a Lie algebra  homomorphsim 
 	$\phi': \Fg^{\phi}\rightarrow\CC$ such that 
 	\beqn
 	yv=\phi'(y)v,~\forall y\in\Fg^{\phi}.
 	\eeqn
 	Since
 	\beqn
 	\phi(p)v=pv=\phi'(p)v,~\forall p\in \Fp,
 	\eeqn
 	the linear map $\phi'$ is an extension of $\phi$. It implies that $\phi$ is extendable.
 \end{proof}
 
 \begin{ex}\label{Examaple3.1}
\emph{ Let $\Fh_1$ be the $3$-dimensional  Heisenberg algebra, which is a complex Lie algebra with basis $x,y,z$ such that
\beqn
[x,y]=z,~[x,z]=[y,z]=0.
\eeqn
Take the ideal $\Fp=\CC z$ and define the Lie algebra homomorphism
\beqn
\phi:\Fp\rightarrow \CC,~\la z\mapsto \la,~\forall \la\in\CC.
\eeqn
   Then $\Fh_1^{\phi}=\Fh_1$. So any nonzero vector in $W(\phi)$ is a Whittaker vector. Note that $\phi$ is not extendable, since
\beqn
\phi([x,y])=\phi(z)=1\neq 0.
\eeqn
By Lemma \ref{L3.6}, $\Fh_1$ has no bland quasi-Whittaker module of type $\phi$. Actually $\Fh_1$ has infinite-dimensional irreducible quasi-Whittaker modules of type $\phi$ all whose vectors are quasi-Whittaker vectors.}
 \end{ex}
 
 Let $\phi:\Fp\rightarrow\CC$ be an extendable Lie algebra homomorphism and $\phi':\Fg^{\phi}\to \CC$ be an extension of $\phi$. Define a one dimensional $\Fg^{\phi}$-module $\CC v_{\phi'}$ by 
 \beqn
 yv_{\phi'}=\phi'(y)v_{\phi'},~\forall y\in\Fg^{\phi}.
 \eeqn
 Let 
 \beqn
 V(\phi')=\Ind_{\Fg^{\phi}}^{\Fg}\CC v_{\phi'}.
 \eeqn
 be the corresponding induced module. It is clear that $V(\phi')$ is a quasi-Whittaker module generated by a quasi-Whittaker vector $v_{\phi'}:=1\otimes v_{\phi'}$ of type $\phi$.  Hence, there exists a surjective $\Fg$-module homomorphism $\Phi: W(\phi) \rightarrow V(\phi')$ sending $w_{\phi}$ to $v_{\phi'}$.  
 
\begin{lem}\label{L3.4} The $\Fg$-module
	$V(\phi')$ is   irreducible.
\end{lem}

\begin{proof} Consider the above surjective $\Fg$-module homomorphism $\Phi: W(\phi) \rightarrow V(\phi')$. 	Let $M$ be a nonzero submodule of $V(\phi')$. It is enough to show that $M=V(\phi')$, or $v_{\phi'}\in M$. 

 Take $0\neq v\in M$. Then $v$ can be uniquely written as 
	\beqn
	v=\sum_{\a\in\ZZ_+^I}a_{\a}x^\a v_{\phi'}
	\eeqn
	with finitely many nonzero $a_{\a}$'s.  Let
	\beqn
	w=\sum_{\a\in\ZZ_+^I}a_{\a}x^\a w_{\phi}\in W(\phi).
	\eeqn
It is easy to see that vectors in $\U(\Fp)w$ are linear combinations of vectors of the form $x^\beta w_{\phi}$ for some $\beta\in \mathbb{Z}^I$.
	  By Proposition \ref{P3.2}, there exists a nonzero quasi-Whittaker vector in 
$\U(\Fp)w\subseteq W(\phi)$ which must be $w_{\phi}$. Thus $v_{\phi'}=\Phi(w_{\phi})\in M$. Hence, $V(\phi')$ is irreducible.
\end{proof}

From the above proof it is clear that the $\Fg$-module $V(\phi')$ is bland of type $\phi$.

\begin{lem}
	Assume that $\phi:\Fp:\rightarrow\CC$ is extendable. Then each bland irreducible quasi-Whittaker module of type $\phi$ is isomorphic to $V(\phi')$ for some extension $\phi'$  of $\phi$. 
\end{lem}

\begin{proof}
	Let $V$ be a bland irreducible quasi-Whittaker module generated by a   quasi-Whittaker vector $v$ of type $\phi$.  From the proof of Lemma \ref{L3.6},  there exists a Lie algebra homomorphism $\phi':\Fg^{\phi}\rightarrow\CC$ extending $\phi$ such that
	 	\beqn
	yv=\phi'(y)v,~\forall y\in\Fg^{\phi}.
	\eeqn
	Form the construction of $V(\phi')$,   there exists a surjective $\Fg$-module homomorphism $\Phi':V(\phi')\rightarrow V$ sending $v_{\phi'}$ to $v$. Since $V(\phi')$ is irreducible, $\Phi'$ must be an isomorphism. 
\end{proof}

The following result is clear.

\begin{lem}
			Let $\phi',\phi'':\Fg^{\phi}\rightarrow\CC$ be two Lie algebra homomorphisms extending $\phi$. Then $V(\phi')$ is isomorphic to $V(\phi'')$ if and only if $\phi'=\phi''$.
\end{lem}

To summarize the above discussions, we obtain the following theorem.

\begin{thm}\label{T3.11} Let $\Fg$ be nonsemisimple Lie algebra,  $\mathfrak{p}$ a nonperfect ideal  of $\Fg$, 
and $\phi:\Fp\rightarrow\CC$ a Lie algebra homomorphism.  Then $\Fg$ admits bland irreducible quasi-Whittaker modules if and only if $\phi$ is extendable. In this case, the correspondence $\phi'\mapsto V(\phi')$ gives a bijection between the set of extensions of $\phi$ and the set of isomorphism class of bland quasi-Whittaker modules of type $\phi$.
\end{thm}
 
\begin{rem}
	\emph{Assume that $\Fg$ is of the form $\Fg=\Fg_0\ltimes \Fp$. Then it is easy to see that each Lie algebra homomorphism from $\Fp$ to  $\CC$ is extendable.  One may expect that each irreducible quasi-Whittaker module is bland. Unfortunately, this is not correct in general.  For example, let $\Fg=\Fgl_n(\CC)=\Fsl_n(\CC)\oplus\CC c$. Let $V$ be an irreducible $\Fsl_n$-module with $\dim V>1$. Then $V$ can be viewed as an irreducible $\Fgl_n(\CC)$-module by declaring $cv=v$ for any $v\in V$.  It is clear that $V$ can be treated  as a quasi-Whittaker module of type 
	\beqn
	\phi:\CC c\rightarrow \CC,~c\mapsto 1,
	\eeqn
	and any element in $V$ is a quasi-Whittaker vector. In order to have good properties for quasi-Whittaker modules over $\Fg$, we generally do not take $\Fp$ to be in the center of $\Fg$.}
\end{rem} 

Finally, we consider a very special case, which is just the situation in \cite{CC,CW}.

\begin{thm}\label{T3.10} Let $\Fg$ be nonsemisimple Lie algebra,  $\mathfrak{p}$ a nonperfect ideal  of $\Fg$, 
and $\phi:\Fp\rightarrow\CC$ a Lie algebra homomorphism.
	Assume that    $\dim \Fg^{\phi}/\Fp\le 1.$ Then any irreducible quasi-Whittaker $\Fg$-module of type $\phi$ is bland.
\end{thm}

\begin{proof}
	Let $V$ be an irreducible quasi-Whittaker module of type $\phi$.
	If $\dim \Fg^{\phi}/\Fp=0$, then $V$ is isomorphic to $W(\phi)$ for the Lie algebra homomorphism $\phi:\Fp\rightarrow\CC$. By Theorem \ref{T3.1}, 
	\beqn
	W(\phi)_{\phi}=\U(\Fg^{\phi})w_{\phi}=\U(\Fp)w_{\phi}=\CC w_{\phi}.
	\eeqn
	It follows that $V$ is bland.
	
	If $\dim \Fg^{\phi}/\Fp=1$,  one can take some $y\in\Fg^{\phi}$ such that $\Fg^{\phi}=\CC y\oplus \Fp.$  Take a nonzero quasi-Whittaker vector $v\in V$.   Then
    there exists a surjective $\Fg$-module homomorphism $\Phi:W(\phi)\rightarrow V$ such that $\Phi(w_{\phi})=v$. By Theorem \ref{T3.5}, $W(\phi)$ is reducible. Hence, $\ker\Phi$ is a maximal submodule of $W(\phi)$. By Proposition \ref{P3.2}, there exists a nonzero quasi-Whittaker vector $w\in \ker\Phi$. By Theorem \ref{T3.1}, there exists some nonzero polynomial $f(y)$  such that $w=f(y)w_{\phi}$. Take such an $f(y)$ with minimal degree and assume that  
    \beqn
    f(y)=a(y-a_1)(y-a_2)\cdots(y-a_s).
    \eeqn
    Then 
    $    	(y-a_1)(y-a_2)\cdots(y-a_s)v=0,$ and  
    	\beqn
    	v'=(y-a_{1})(y-a_2)\cdots(y-a_{s-1})v\in V\setminus\{0\}.
    	\eeqn
    	Then $v'$ is a nonzero quasi-Whittaker vector of type $\phi$ such that $yv'=a_sv'$. Define 
    	$\phi':\Fg^{\phi}\rightarrow \CC$
    	by setting $\phi'(y)=a_s$. Then $\phi'$ is a Lie algebra homomorphism extending $\phi$. By the construction of $V(\phi')$, we see that $V$ is a homomorphic image of $V(\phi')$.  Since $V(\phi')$ and $V$ are irreducible, we must have $V\cong V(\phi')$.  It follows that $V$ is bland.
\end{proof}

From the above proof we obtain the following

\begin{cor}\label{C3.12} Let $\Fg$ be nonsemisimple Lie algebra,  $\mathfrak{p}$ a nonperfect ideal  of $\Fg$, 
and $\phi:\Fp\rightarrow\CC$ a Lie algebra homomorphism.
	Assume that $\dim \Fg^{\phi}/\Fp\leq 1$.  
		\begin{enumerate}
		\item[(1)] If $\dim \Fg^{\phi}/\Fp = 0$, then  $W(\phi)$ is irreducible over $\Fg$.
		\item[(2)] If $\dim \Fg^{\phi}/\Fp=1$, taking $y\in\Fg^{\phi}$ with $\Fg^{\phi}=\CC y\oplus \Fp$, then all maximal submodules of $W(\phi)$ are of the form $$J_{\xi}:=\mathrm{Span}_{\mathbb{C}}\{x^{\alpha}y^n(y -\xi)w_{\phi}\mid\alpha\in\mathbb{Z}_+^I,n\in\mathbb{Z}_+\},\quad \text{where } \xi\in\mathbb{C}.$$
		In fact, we have $V(\phi')\cong W(\phi)/J_{\phi'(y)}$ for any extension $\phi'$ of $\phi$.
	\end{enumerate}
\end{cor}

\section{Applications}
In this section, we will apply the main results established in Section 3 to classify   irreducible quasi-Whittaker modules over many   Lie algebras to recover many known rersults, including Conformal Galilei algebra and $n$-th Schrödinger algebra.  Then we also obtain a lot of irreducible quasi-Whittaker modules over mirror Heisenberg-Virasoro algebra, Heisenberg-Virasoro algebra, and Planer Galilei Conformal algebra. In addition,  we obtain a lot of irreducible  smooth $\W_n^+$-modules of height $2$.

\subsection{Conformal Galilei algebra and $n$-th Schrödinger algebra}

Following \cite{MOR}, let $\ell\in \frac{1}{2}\mathbb{N}$ and   $\overline{0,2\ell} = \{0,1,\ldots,2\ell\}$. The (centreless) conformal Galilei algbera $\Fg^{(\ell)}$ is a complex Lie algebra with basis $\{e, h, f, p_k\mid k \in \overline{0,2\ell}\}$.  The non-vanishing Lie brackets for $\Fg^{(\ell)}$ are given by
\begin{align*}
	&[h, e] = 2e,~~~~~~~~~~~~~~ [h, f]= -2f, ~~~~~                 [e, f] = h,\\
	&[h, p_k]= 2(l - k)p_k,~~~ [e, p_k]= kp_{k - 1}, ~~~[f, p_k]= (2l - k)p_{k + 1}.
\end{align*}
The subalgebra of $\Fg^{(\ell)}$ spanned by $e,f, h$ is isomorphic to   $\Fsl_2(\CC)$.  Let 
\beqn
\Fp=\Span_{\CC}\{ p_k\mid k\in\overline{0,2\ell}\},
\eeqn
which is an ideal of $\Fg^{(\ell)}$. As a $\Fsl_2(\CC)$-module, $\Fp$ is isomorphic to the irreducible one with dimension $2\ell+1$. We have a decomposition
\beqn
\Fg^{(\ell)}=\Fsl_2(\CC)\ltimes \Fp.
\eeqn
Let $\phi: \Fp\rightarrow \CC$ be a nonzero Lie algebra homomorphism. For any two-dimensiomal subalgebra $\mathfrak{a}$ of $\Fsl_2(\CC)$, we know that 
$[\mathfrak{a},\Fp]=\Fp$.
 Since $\phi\neq0$, we must have
 $\phi\([\mathfrak{a},\mathfrak{p}]\)\neq 0$.  It follows that $\dim \Fg^{(\ell)\phi}/\Fp\le1$.

Following \cite{T}, the $n$-th Schrödinger algebra $\mathfrak{sch}_n$ is a complex Lie algebra with basis $\{h,e,f,x_i,y_i,z\mid i = 1,\ldots,n\}$ and the Lie brackets
\begin{align*}
	&[h,e]=2e,\quad\quad\quad\quad\quad [h,f]= - 2f,\quad [e,f]=h,\\
	&[h,x_i]=x_i,\quad\quad\quad\quad~~ [e,x_i]=0,\quad \quad~[f,x_i]=y_i,\\
	&[h,y_i]=-y_i,\quad\quad\quad\quad [e,y_i]=x_i,\quad\quad [f,y_i]=0,\\
	&[x_i,x_j]=[y_i,y_j]=0,~~[x_i,y_j]=\delta_{ij}z,\quad [z,\cdot]=0.
\end{align*}
Note that the subspace spanned by $h,e,f$ is in fact a subalgebra, which is isomorphic to $\Fsl_2(\CC)$ canonically. Let
\beqn
\Fh_n=\Span_{\CC}\{x_i,y_i,z\mid 1\le i\le n\}.
\eeqn
Then $\Fh_n$ is isomorphic to the $n$-th Heisenberg Lie subalgebra. Moreover, we have 
\beqn
\mathfrak{sch}_n=\Fsl_2(\CC)\ltimes \Fh_n.
\eeqn
We choose $\Fp=\Fh_n$. Let $\phi:\Fp\rightarrow\CC$ be a nonzero Lie algebra homomorphism (implying that  $\phi(z)=0$). For any two-dimensiomal subalgebra $\mathfrak{a}$ of $\Fsl_2(\CC)$, it is easy to verify that  $\Fp=[\mathfrak{a},\Fp]\oplus\CC z$. Since $\phi\neq0$ and $\phi(z)=0$, we must have $\phi\([\mathfrak{a},\mathfrak{p}]\)\neq0$. 
It follows that $\dim\mathfrak{sch}_n^{\phi}/\Fp\le 1$. 

From above the discussions, using Corollary \ref{C3.12} we can deduce the following theorem, which was obtained in \cite{CC,CW} by concrete calculations.

\begin{thm}
	Let $\Fg$ be the conformal Galilei algebra $\Fg^{(\ell)}$ or $n$-th Schrödinger algebra $\mathfrak{sch}_n$. Assume that $\phi:\Fp\rightarrow \CC$ is a nonzero Lie algebra homomorphism. Then each irreducible quasi-Whittaker module of type $\phi$ is bland, i.e., is of the form of either $W(\phi)$ or $V(\phi')$ for some extension $\phi'$ of $\phi$.
\end{thm}

\begin{rem}
	When  $n = \infty$, we can  call the corresponding Schrödinger algebra as the $\infty$-th Schrödinger algebra. In fact, the conclusion of Theorem 4.1 also holds  for the $\infty$-th Schrödinger algebra.
\end{rem}

\subsection{The mirror Heisenberg-Virasoro algebra}

Following \cite{Ba}, the mirror Heisenberg-Virasoro algebra $\mathcal{D}$ is a Lie algebra with a basis
\[
\left\{ d_m, h_r, \mathbf{c}, \mathbf{l} \,\bigg|\, m \in \mathbb{Z}, r \in \frac{1}{2} + \mathbb{Z} \right\}
\]
subject to the following commutation relations:
\begin{align*}
	&[d_m, d_n] = (m - n)d_{m+n} + \frac{m^3 - m}{12} \delta_{m+n,0} \mathbf{c},  \\
	&[d_m, h_r] = -r h_{m + r}, \quad [h_r, h_s] = r \delta_{r + s, 0} \mathbf{l}, \quad
	[\mathbf{c}, \mathcal{D}] = [\mathbf{l}, \mathcal{D}] = 0,
\end{align*}
for $m, n \in \mathbb{Z}$ and $r, s \in \frac{1}{2} + \mathbb{Z}$. Let
\begin{equation*}
	\mathcal{V}=\Span_{\CC}\{d_m,\mathbf{c}\mid m\in\ZZ\},~\mathcal{H}=\Span_\CC\{h_r, \mathbf{l} \mid r \in \frac{1}{2} + \mathbb{Z}\}.
\end{equation*}
Then $\mathcal{V}$ is the Virasoro algebra and $\mathcal{H}$ is the twisted Heisenberg algebra. Moreover, one has $\mathcal{D}=\mathcal{V}\ltimes \mathcal{H}.$ We take $\Fp=\CC \mathbf{c}\oplus\mathcal{H}$. Then $\Fp$ is an ideal of $\mathcal{D}$. 

	Let $\phi:\mathfrak{p}\to\mathbb{C}$ be a Lie algebra homomorphism. Since $[\mathfrak{p},\mathfrak{p}]=\mathbb{C}\mathbf{l}$, then $\phi(\mathbf{l}) = 0$. Following \cite{BZ}, we say $\phi$ is an exp-polynomial function if there exists nonzero distinct complex numbers $a_1,a_2,\cdots,a_\ell$ and polynomials $f_1(t),f_2(t),\cdots,f_\ell(t)\in\mathbb{C}[t]$ such that
	\[
	\phi(h_{\frac{1}{2}+n})=a_1^nf_1(n)+a_2^nf_2(n)+\cdots+a_\ell^nf_\ell(n),\quad\forall n\in\mathbb{Z}.
	\]
	The following lemma can be deduced from \cite[Lemma 3.3]{CGT}.
	\begin{lem}\label{lem4.2}
		Let $\phi:\mathfrak{p}\to\mathbb{C}$ be a   linear map  such that $\phi(\mathbf{c})=0$. Then $\phi$ is an exp-polynomial function if and only if there exists some $c_0,c_1,\cdots,c_{r - 1}\in\mathbb{C}$ with $c_0\neq0, r\ge1$ such that
		\[
		c_0\phi(h_{\frac{1}{2}+n})+c_1\phi(h_{\frac{1}{2}+n + 1})+\cdots+c_{r-1}\phi(h_{\frac{1}{2}+n + r-1})+\phi(h_{\frac{1}{2}+n + r}) = 0,\quad\forall n\in\mathbb{Z}.\tag{4.1}
		\]
	\end{lem}

\begin{thm}
	Let $\phi:\Fp\rightarrow\CC$ be a  Lie algbera homomorphism  {with $\phi(\mathcal{H})\ne0$}. Then $W(\phi)$ is irreducible if and only if  $\phi$ is not an exp-polynomial function. 
\end{thm}

\begin{proof}
	Assume that $\phi$ is not an exp-polynomial function. Suppose that $\mathcal{D}^\phi\neq\Fp$, that is,  there exists nonzero
	\beqn
	x=c_0d_{m-r}+c_1d_{m-r+1}+\cdots+c_{r-1}d_{m-1}+d_m\in\mathcal{D}^{\phi}.
	\eeqn
	Then 
	\beqn
	0=\phi([x,h_{\frac{1}{2}+n}])=-\(\frac{1}{2}+n\)\(c_0\phi(h_{\frac{1}{2}+n+m-r})+\cdots+c_{r-1}\phi(h_{\frac{1}{2}+n+m-1})+\phi(h_{\frac{1}{2}+n+m})\)
	\eeqn
	holds for each $n\in\ZZ$, thus
	\beqn
	c_0\phi(h_{\frac{1}{2}+n-r})+c_1\phi(h_{\frac{1}{2}+n-r+1})+\cdots+c_{r-1}\phi(h_{\frac{1}{2}+n-1})+\phi(h_{\frac{1}{2}+n})=0,~\forall n\in\ZZ.
	\eeqn
	By Lemma \ref{lem4.2}, $\phi$ must be an exp-polynomial function contadicting the assumption. So $\mathcal{D}^\phi=\Fp$, and $W(\phi)$ is irreducible by Theorem \ref{T3.5}.
	
	Assume that $W(\phi)$ is irreducible, i.e. $\mathcal{D}^\phi=\Fp$  by Theorem \ref{T3.5}. Suppose that $\phi$ is an exp-polynomial function. Then by Lemma \ref{lem4.2}, there 
	exists some $c_0,c_1,\cdots,c_{r-1}\in\CC$ with $c_0\neq 0$ such that
	\beqn
	c_0\phi(h_{\frac{1}{2}+n-r})+c_1\phi(h_{\frac{1}{2}+n-r+1})+\cdots+c_{r-1}\phi(h_{\frac{1}{2}+n-1})+\phi(h_{\frac{1}{2}+n})=0,~\forall n\in\ZZ.
	\eeqn
	It follows that 
	\beqn
	c_0d_{-r}+c_1d_{-r+1}+\cdots+c_{r-1}d_{-1}+d_0\in\mathcal{D}^{\phi},
	\eeqn
	a contradiction.  So $\phi$ is not an exp-polynomial function.
\end{proof}

\subsection{The Heisenberg-Virasoro algebra}

 Following \cite{ADKP}, the Heisenberg-Virasoro algebra $\mathcal{L}$ is a Lie algebra with a basis $\{L_n, I_n, z_1, z_2, z_3\mid n\in\mathbb{Z}\}$ defined by the following commutation relations
  \begin{align*}
 	&[L_m,L_n]=(n - m)L_{m + n}+\frac{m^3 - m}{12}\delta_{m + n,0}z_1, \quad [I_m,I_n]=m\delta_{m + n,0}z_3,\\
 	&[L_m,I_n]=nI_{m + n}+\delta_{m + n,0}(m^2 + m)z_2, \quad [I_0,\mathcal{L}]=[z_1,\mathcal{L}]=[z_2,\mathcal{L}]=[z_3,\mathcal{L}]=0.
 \end{align*}
 Let
\begin{equation*}
	\mathcal{V}=\Span_{\CC}\{L_m,z_1\mid m\in\ZZ\},~\mathcal{H}=\Span_\CC\{I_r, z_3 \mid r \in  \mathbb{Z}\}.
\end{equation*}
Then $\mathcal{V}$ is the Virasoro algebra and $\mathcal{H}$ is the  Heisenberg algebra. 
 We take the ideal $\Fp$ to be
 \beqn
 \Fp=\Span_{\CC}\{I_n,z_1,z_2,z_3\mid n\in\ZZ\}.
 \eeqn
 Then $[\Fp,\Fp]=\CC z_3$. Take a   Lie algbera homomorphism $\phi:\Fp\to \mathbb{C}$, implying that  $\phi(z_3)=0$.  We set
 \beqn
 S^\phi=\{n\in\ZZ\mid \phi(I_n)\neq 0\}.
 \eeqn
 We call $\phi$ {\bf finite} if $|S^{\phi}|<+\infty$ and $\phi(z_2)=0$. 
 
\begin{ex}\label{ex41}
	For a fixed $k\in \mathbb{Z}$, let $\phi:\Fp\rightarrow \CC$ be a Lie algbera homomorphism such that
	\beqn
	\phi(I_k)=1,~\phi(I_n)=0,~\forall n\neq k; \phi(z_1)=\phi(z_2)=\phi(z_3)=0.
	\eeqn
	Then $\mathcal{L}^\phi=\CC L_k\oplus\Fp.$ In addition, each irreducible quasi-Whittaker module of type $\phi$ is of the form of  $V(\phi')$ for some extension $\phi'$ of $\phi$.
\end{ex}
\begin{proof}
	Assume that 
	\beqn
	x=L_m+\la_1 L_{m+1}+\cdots+\la_{\ell}L_{m+\ell}\in \Fg^{\phi}.
	\eeqn
	Then
	\beqn
	m-k=-\phi([x,I_{-m+k}])=0.
	\eeqn
	Thus, $x$ is expressed as
	\beqn
	x=L_k+\la_1 L_{m+1}+\cdots+\la_{\ell}L_{m+\ell}.
	\eeqn
	For $1\le i\le \ell$, one has
	\beqn
	\la_i=-\frac{1}{i}\phi([x,I_{-i}])=0.
	\eeqn
	We can obtain that 
	\beqn
	\mathcal{L}^\phi=\CC L_k\oplus\Fp.
	\eeqn
Then, according to Corollary \ref{C3.12}, the proof is completed.
\end{proof}
 
 \begin{thm}\label{HV}
 	Let $\phi:\Fp\rightarrow\CC$ be  a finite Lie algbera homomorphism {with $\phi(\mathcal{H})\ne0$}. Then $W(\phi)$ is irreducible  if and only if $|S^{\phi}|\ge2$.
 \end{thm}
 
 \begin{proof}
 	Assume that $|S^{\phi}|\ge2$. Let $i$ be the  smallest element in $S^{\phi}$ and $j$ the largest one. Suppose  that $\mathcal{L}^{\phi}\neq \Fp$. Then one can choose 
 	\beqn
 	x=L_m+\sum_{k=1}^{+\infty}\la_k L_{m+k}\in \mathcal{L}^{\phi},
 	\eeqn
 	where $\la_k=0$ for almost all $k\geq 1$. We claim that $m\le i$. Otherwise, one has
 	\beqn
 	0=\phi([x,I_{i-m}])=(i-m)\(\phi(I_i)+\sum_{k=1}^{\infty}\la_k\phi(I_{i+k})\),
 	\eeqn
 	thus
 	\beqn
 	\phi(I_i)+\sum_{k=1}^{+\infty}\la_k\phi(I_{i+k})=0.
 	\eeqn
 	It is clear that there exists some $t_1\geq 1$ such that $\la_{t_1}\neq 0$. Again, we have
 	\beqn
 	0=\phi([x,I_{i-m-t_1}])=(i-m-t_1)\(\la_{t_1}\phi(I_{i})+\sum_{k=t_1+1}^{+\infty}\la_k\phi(I_{i+k-t_1})\),
 	\eeqn
 	thus
 	\beqn
 	\la_{t_1}\phi(I_i)+\sum_{k=t_1+1}^{+\infty}\la_k\phi(I_{i+k-t_1})=0.
 	\eeqn
 	We can find $t_2>t_1$ such  that $\la_{t_2}\neq 0$. Repeating the above process, we finally obtain a sequence of nonzero numbers 
 	$\la_{t_1},\la_{t_2},\la_{t_3},\cdots,$ a contradiction.  Now we have $m\le i<j$. But this yields that
 	\beqn
 	\phi([x,I_{j-m}])=(j-m)\phi(I_j)\neq 0,
 	\eeqn
 	also a contradiction.  Hence, we must have $\mathcal{L}^{\phi}=\Fp$, which implies that $W(\phi)$ is irreducible.
 	
 	Next, we assume that $W(\phi)$ is irreducible. Then $\mathcal{L}^{\phi}=\Fp$. If $|S^{\phi}|=0$, i.e., $S^{\phi}=\emptyset$, then $\mathcal{L}^\phi=\mathcal{L}\supsetneq \Fp$, a contradiction.  If $|S^{\phi}|=1$, say $S^{\phi}=\{n\}$, then it is easy to check that $L_n\in\mathcal{L}^{\phi}$, which implies that $\mathcal{L}^\phi\supsetneq \Fp$, a contradiction (Also see Example \ref{ex41} ). The proof is completed.
 \end{proof}

\subsection{Planer Galilei Conformal algebra}

 Following \cite{BG}, the planer Galilean conformal algebra  $\mathcal{G}$ is a Lie algebra with basis $\{L_n, H_n, I_n, J_n \mid n \in \mathbb{Z}\}$ (different from that in \cite{BG}) defined by the following commutation relations
 \begin{align*}
	&[L_n, L_m] = (m - n)L_{m+n},  [L_n, H_m] = mH_{m+n}, \notag \\
	&[L_n, I_m] = (m - n)I_{m+n},  [L_n, J_m] = (m - n)J_{m+n}, \notag \\
	&[H_n, I_m] = I_{m+n},  [H_n, J_m] = -J_{m+n}, \notag \\
	&[H_n, H_m] = [I_n, I_m] = [J_n, J_m] = [I_n, J_m] = 0, \quad \forall m, n \in \mathbb{Z}. \label{eq:2.4}
\end{align*}

Let 
\beqn
\Fp=\Span_{\CC}\{H_n,I_n,J_n\mid n\in\ZZ\}.
\eeqn
Then $[\Fp,\Fp]=\Span_{\CC}\{I_n,J_n\mid n\in\ZZ\}$. Hence, for any Lie algebra homomorphism $\phi:\Fp\to \mathbb{C}$  we have $\phi(I_n)= \phi(J_n)=0$ for any $n \in \mathbb{Z}$.  We set
\beqn
S^\phi=\{n\in\ZZ\mid \phi(H_n)\neq 0\}.
\eeqn
We call $\phi$ {\bf finite} if $|S^{\phi}|<+\infty$ . Since $I$ and $J$ are ideals of  $\mathcal{G}$, for the $\mathcal{G}$-module $W(\phi)$, we have $(I + J). W(\phi) = 0$. Therefore, the $\mathcal{G} $-module $W(\phi)$ can be regarded as a $\mathcal{G}/(I + J)$-module while preserving the same irreducibility. Consequently, we obtain a result analogous to Theorem \ref{HV} for the Heisenberg-Virasoro algebra, namely:

\begin{thm}
	Let $\phi:\Fp\rightarrow\CC$ be  a finite Lie algbera homomorphism. Then $W(\phi)$ is irreducible  if and only if $|S^{\phi}|\ge2$.
\end{thm}

We certainly can use our theory in Sect.3 to  get a lot of simple $\mathcal{G}$-modules by taking 
$
\mathfrak{p}=\Span_{\CC}\{I_n,J_n\mid n\in\ZZ\}.
$

\subsection{The Lie algebra $\mathcal{W}(a,b)$}
  Following \cite{OR}, the Lie algebra $\W(a,b)$ is the semi-direct product of centerless Virasoro $\mathrm{Vir}$ and its intermediate series module $A_{a,b}$, i.e., $\W(a,b)=\mathrm{Vir}\ltimes A_{a,b}$ where $a,b\in\mathbb C$. As a vector space, $\W(a,b)$ has a basis $\{L_i,H_i\mid i\in\ZZ\}$ with Lie brackets
  \beqn
  [L_i,L_j]=(j-i)L_{i+j},~[L_i,H_j]=(a+j+bi)H_{i+j},~[H_i,H_j]=0, \forall, i, j\in\ZZ.
  \eeqn
  Note that $\W(0,-1)$ is just the Takiff algebra of the centerless Virasoro algebra.  
  We abbreviate 
  \beqn
  \Fg=\W(a,b),~\Fg_0=\Span_{\CC}\{L_i\mid i\in\ZZ\},~ \Fp=\Span_{\CC}\{H_i\mid i\in\ZZ\}.
  \eeqn 
  Since $[\Fp,\Fp]=0$,  each $\phi\in\Fp^*$ is a Lie algebra homomorphism from $\Fp$ to $\CC$. For $\phi\in\Fp^*$, we write
  \beqn
  S^{\phi}=\{j\in\ZZ\mid \phi(H_j)\neq0\}.
  \eeqn
  We call $\phi\in\Fp^*$ {\bf upper} (resp. {\bf lower}) {\bf finite} if there exists $j_0\in\ZZ$  such that  $\phi(H_j)=0$   for each $j>j_0$ (resp. $j<j_0$). 
  
   Let $x$ be a nonzero element in $\Fg_0$. Then $x$ can be uniquely expressed as
 \beq\label{F4.1}
 x=\la_iL_i+\la_{i+1}L_{i+1}+\cdots+\la_jL_j
 \eeq
 with $j\ge i$ and $\la_i\la_j\neq0$.  Define the {\bf head} (resp. {\bf tail}) of $x$ to be $\mathrm{h}(x)=i$ (resp. $\mathrm{t}(x)=j$).
 
 \begin{lem}\label{L4.5.1}
 	Let $\phi:\Fg\rightarrow\CC$ be a nonzero upper (resp. lower) finite Lie algbera homomorphism and $j_0$ the maximal (resp. miminal) element in $S^{\phi}$ and $x\in\Fg^{\phi}\cap\Fg_0$ a nonzero element.  Then
 	\beqn
 	a+j_0=(1-b)\mathrm{h}(x) ~~(\text{resp.}~ a+j_0=(1-b)\mathrm{t}(x) ).
 	\eeqn
 \end{lem}
 
 \begin{proof}
 	We only prove the case that $\phi$ is upper finite. Assume that $x$ is shown as \eqref{F4.1}. Then 
 	\beqn
 	0=\phi\([x,H_{j_0-i}]\)=\la_i\(a+j_0-i+bi\)\phi(H_{j_0}).
 	\eeqn
 	Since $\la_i$ and $\phi(H_{j_0})$ are both nonzero, we can obtain $a+j_0-i+bi=0$, i.e. 
 	\beqn
 	a+j_0=(1-b)\mathrm{h}(x).
 	\eeqn
 	The proof is complete.
 \end{proof}

 \begin{prop}
 	Assume that $b=1$. Let $\phi:\Fg\rightarrow\CC$ be a nonzero upper (resp. lower) finite Lie algbera homomorphism and $j_0$ the maximal (resp. miminal) element in $S^{\phi}$. Then $W(\phi)$ is reducible if and only if $S^{\phi}=\{j_0\}$ and $a=-j_0$.
 \end{prop}
 
 \begin{proof}
 	Assume that $S^{\phi}=\{j_0\}$ and $a=-j_0$. It is easy to check that $\Fg^{\phi}=\Fg$, which follows that $W(\phi)$ is reducible.
 	
 	Next, we assume that $W(\phi)$ is reducible. Then $\Fp\subsetneq\Fg^{\phi}$. One can choose a nonzero $x\in \Fg^{\phi}\cap\Fg_0$. Assume that $x$ is shown as \eqref{F4.1}. Then $\mathrm{h}(x)=i$ By Lemma \ref{L4.5.1}, one has
 	\beqn
 	a+j_0=(1-b)\mathrm{h}(x)=(1-1)\mathrm{h}(x)=0.
 	\eeqn
 	We have
 	\beqn
 	\phi\([L_k,H_{j_0-k}]\)=(a+j_0)\la_i\phi\(H_{j_0}\)=0,~\forall k\in\ZZ.
 	\eeqn
 	Since
 	\beqn
 	0=\phi\([x,H_{j_0-(i+1)}]\)=(a+j_0-1)\la_i\phi(H_{j_0-1}),
 	\eeqn
 	we have $\phi(H_{j_0-1})=0$. By induction on $m$, it is easy to obtain that 
 	\beqn
 	\phi(H_{m})=0,~\forall m<j_0.
 	\eeqn
 	Thus, $S^{\phi}=\{j_0\}$. 
 \end{proof}

\begin{prop}\label{P4.8}
	Assume that $b\neq 1$. Let $\phi:\Fg\rightarrow\CC$ be a nonzero upper (resp. lower) finite Lie algbera homomorphism and $j_0$ the maximal (resp. miminal) element in $S^{\phi}$.  Then  $W(\phi)$ is reducible only if $\frac{j_0+a}{1-b}\in\ZZ$. Moreover, each irreducible quasi-Whittaker module of type $\phi$ is bland.
\end{prop}

\begin{proof}
	Assume that  $W(\phi)$ is reducible. Then $\Fg^{\phi}\cap \Fg_0\neq 0$. Let $x$ be a nonzero element in  $ \Fg^{\phi}\cap \Fg_0$. By lemma \ref{L4.5.1}, one has
	\beqn
	\frac{j_0+a}{1-b}=\mathrm{h}(x)\in\ZZ.
	\eeqn
	It follows that all the nonzero elements in  $ \Fg^{\phi}\cap \Fg_0$ have a same head. This implies that  $\Fg^{\phi}=\CC x\oplus\Fp$. By Theorem \ref{T3.10}, each irreducible irreducible quasi-Whittaker module of type $\phi$ is bland.
\end{proof}

  \begin{ex}
  	Let $\Fg=\W(0,-1)$ be the Takiff algebra of the centerless Virasoro algebra and $j_0\in\ZZ$  an odd number. Assume that  $\phi:\Fp\rightarrow\CC$ is a Lie algebra homomorphism with $\phi(H_{j_0})\neq 0$ and $\phi(H_j)=0$ for each $j>j_0$ (resp. $j<j_0$). Then $W(\phi)$ is irreducible.
  \end{ex}
  
  \begin{proof}
  	It is clear that $\phi$ is upper (resp. lower) finite. Since 
  	\beqn
  	\frac{0+j_0}{1-(-1)}=\frac{j_0}{2}\notin\ZZ,
  	\eeqn
  	by Proposition \ref{P4.8}, we have $W(\phi)$ is irreducible.
  \end{proof}

\subsection{Borel subalgebra $\mathcal{W}_1^{++}$ of the Witt algebra $\mathcal{W}_1$}
The following conclusion will recover the results in \cite{LGZ}.
Let $\mathfrak{g}= \mathcal{W}_1^{++}=\text{Span}_{\mathbb{C}}\{d_i \mid \ i\in \mathbb{Z}_+\}$ be the Borel subalgebra of the Witt algebra $\W_1 = \text{Span}_{\mathbb{C}}\{d_i \mid \ i\in \mathbb{Z}\}$. The Lie brackets in $\mathfrak{g}$ are given by 
$$[d_m,d_n]=(m-n)d_{m+n}, \ \text{for any} \ m,n \in \mathbb{Z}_+.$$
Fix a $k\in \mathbb{N}$, let $\mathfrak{p}_{k}= \text{Span}_{\mathbb{C}}\{d_i \mid \ i\in \mathbb{Z}_{\geq k}\}$ be the ideal of the Lie algebra $\mathfrak{g}$. Let $\phi_k: \Fp_k\rightarrow \CC$ be a nonzero Lie algebra homomorphism, then we have $\phi(\mathfrak{p}_{2k+1})=0$ from  $[\mathfrak{p}_k,\mathfrak{p}_k]=\mathfrak{p}_{2k+1}$.

\begin{thm}
	Let $\phi_k: \Fp_k\rightarrow \CC$ be a nonzero Lie algebra homomorphism. Then the universal quasi-Whittaker module $W(\phi)$ over $\W_1^{++}$ is irreducible if and only if $\phi_k(d_{2k-1})\neq0$ or $\phi_k(d_{2k})\neq0$.
\end{thm}
\begin{proof}
Suppose that $\phi_k(d_{2k-1})=0$ and $\phi_k(d_{2k})=0$. 
It is easy to check that $d_{k-1}\in\mathfrak{g}^{\phi_k}$. 
Therefore, by Theorem \ref{T3.5}, the $\mathfrak{g}$-module $W(\phi_k)$ is reducible.

Next, choose 
\[
x = \ell_0d_0 + \ell_1d_1 + \cdots + \ell_{k-1}d_{k-1} \in \mathfrak{g}^{\phi_k}.
\]
Assume $\ell_0 \neq 0$. By $\phi_k([x,d_{2k}]) = 0$, we derive $\phi_k(d_{2k}) = 0$. 
Similarly, by $\phi_k([x,d_{2k-1}]) = 0$, we obtain $\phi_k(d_{2k-1}) = 0$. 
By induction, $\phi_k(\mathfrak{p}_k) = 0$, which contradicts the non-vanishing of $\phi_k$. 
Thus, $\ell_0 = 0$, i.e.,
\[
x = \ell_1d_1 + \cdots + \ell_{k-1}d_{k-1} \in \mathfrak{g}^{\phi_k}.
\]

\textbf{Case 1:} $\phi_k(d_{2k}) \neq 0$.

From $\phi_k([x, d_{2k-1}]) = 0$, we get $\ell_1 = 0$. 
By induction on $\phi_k([x, d_{2k-i}]) = 0$ for $i \geq 1$, 
we successively find $\ell_i = 0$ for $1 \leq i \leq k-1$, so $x = 0$.

\textbf{Case 2:} $\phi_k(d_{2k}) = 0$ and $\phi_k(d_{2k-1}) \neq 0$.

Similarly, considering $\phi_k([x, d_{2k-1-i}]) = 0$, induction yields $\ell_i = 0$ 
for $1 \leq i \leq k-1$, hence $x = 0$.

By Theorem \ref{T3.5}, the $\mathfrak{g}$-module $W(\phi)$ is irreducible.
\end{proof}

\subsection{Irreducible smooth modules over $\mathcal{W}_n^+ $ with height 2 }

Let $n\geq2$. Following \cite{LCLLZZ}, Let $A_{n}^{+}=\mathbb{C}[t_{1},t_{2},\dots,t_{n}]$ be the polynomial algebra in variables $t_1,t_2,\dots,t_n$  over $\mathbb{C}$ and let $\W_{n}^{+}$ be the derivation Lie algebra of $A_n^+$.
For $1\leq i\le n$, we denote by $\partial_i=\partial/\partial t_i$,
the partial derivation operator with respect to $t_i$.
For $\alpha=(\alpha_1,\alpha_2,\dots,\alpha_n)\in\mathbb{Z}_+^n$, set
\[t^\alpha=t_1^{\alpha_1}t_2^{\alpha_2}\cdots t_n^{\alpha_n}\ \text{and}\
\partial^\alpha=\partial_1^{\alpha_1}\partial_2^{\alpha_2}\cdots \partial_n^{\alpha_n}  .\]
We have
\[\W_{n}^{+}=\text{Span}_{\mathbb{C}}\{t^\alpha\partial_i\mid \alpha\in\mathbb{Z}_{+}^n,\ i=1,2,\dots,n\}\]
and the Lie brackets in $\W_{n}^{+}$ are given by
\begin{align*}\label{lie-bracket}
	[t^{\alpha}\partial_{i},t^{\beta}\partial_{j}]=t^{\alpha}\partial_{i}(t^{\beta})
	\partial_{j}-t^{\beta}\partial_{j}(t^{\alpha})\partial_{i}.
\end{align*}
In this section, we write $\mathfrak{g}=\W_n^+$ for convenience.
 
Let $\mathfrak{g}_{k}=0$ for $k<-1$ and
\[\mathfrak{g}_k=\text{Span}_{\mathbb{C}}\{t^{\alpha}\partial_{i}\mid \alpha\in\mathbb{Z}_{+}^n,\ |\alpha|=k+1,\ i=1,2,\dots,n\}\]
for $k\geq-1$. This gives a natural $\mathbb{Z}$-gradation on $\mathfrak{g}$:
\[\mathfrak{g}=\bigoplus_{k\in\mathbb{Z}}\mathfrak{g}_k.\]
It is clear that $\mathfrak{g}_{\geq r}$ is a subalgebra of $\mathfrak{g}$ for all $r\in\mathbb{Z}$. We abbreviate $\mathfrak{a}=\mathfrak{g}_{\geq 0}$ and $\mathfrak{p}=\mathfrak{g}_{\geq 1}$, then
\beqn
\mathfrak{a}=\mathfrak{g}_0\ltimes \mathfrak{p}.
\eeqn
Moreover, $[\mathfrak{p},\mathfrak{p}]=\mathfrak{g}_{\geq2}.$ Let $\phi:\mathfrak{p}\rightarrow\mathbb{C}$ be a nonzero Lie algbera homomorphism, i.e. a linear map such that
\beqn
\phi\(\mathfrak{g}_1\)\neq 0,~\text{and}~\phi\(\mathfrak{g}_{\geq2}\)=0.
\eeqn

\begin{Def}
	A $\mathfrak{g}$-module (or $\mathfrak{a}$-module) $V$ is called a {\bf smooth} module
	if for any $v\in V$ there exists $n\in\mathbb{Z}_{+}$ such that $\mathfrak{g}_{i}.v=0$ for all $i\geq n$.
\end{Def}

The concept was first introduced  by D. Kazhdan and G. Lusztig  for   affine Kac-Moody algebras \cite{KL1, KL2} in 1993.

Let $V$ be a smooth $\mathfrak{g}$-module (or $\mathfrak{a}$-module). For $r\in\mathbb{Z}_{+}$, set
\begin{align*}
	V^{(r)}=\{v\in V\mid \mathfrak{g}_{\geq r}.v=0\}.
\end{align*}
Then $V^{(r)}$ is a subspace of $V$ and
\[V^{(0)}\subseteq V^{(1)}\subseteq V^{(2)}\subseteq\cdots.\]
Define $\ell_{V}=\mathrm{min}\{r\in\mathbb{Z}_{+}\mid V^{(r)}\neq 0\}$ which is called the {\bf height} of $V$. Next, we will apply the theories in Sections 2 and 3 to the Lie algebra $\mathfrak{a}$ and its ideal $\mathfrak{p}$, rather than to $(\mathfrak{g}, \mathfrak{p})$.
\begin{prop}
	Keep the notations above. Let $\phi:\mathfrak{p}\rightarrow\mathbb{C}$ be a Lie algbera homomorphism such that 
	\beqn
	\phi(t_i^2\partial_i)\neq 0,~\forall i=1,2,\cdots,n
	\eeqn
	and
	\beqn
	\phi(t^\a\partial_i)=0,~\forall \a\in\mathbb{Z}_+^n \text{~with~} |\a|=2~\text{and}~\a_i\neq 2.
	\eeqn
	Then $W(\phi)$ is irreducible. In particular, $V=\mathrm{Ind}_{\mathfrak{a}}^{\mathfrak{g}}W(\phi)$ is an irreducible smooth $\mathfrak{g}$-module of height $2$.
\end{prop}

\begin{proof}
	Denote by
	\beqn
	I=\{(1,1)<(1,2)<\cdots<(1,n)<\cdots<(n,1)<(n,2)<\cdots<(n,n)\}.
	\eeqn
	It is elementary to verify that  $\phi\([t_i\partial_j,t_p^2\partial_q]\)\neq0$ if and only if $j=p$ and $i=q$ for $1\leq i,j,p,q\leq n$. It implies that the matrix
	\beqn
	D=\(\phi([t_i\partial_j,t_p^2\partial_q])\)_{(i,j)\in I,~(p,q)\in I}
	\eeqn
	has exactly one non-zero element in each row and each column. In particular,
	$\det D\neq 0$, which provides a nonzero minor of $A_{\phi}$ of order $n^2$.  It follows that $\mathfrak{a}^{\phi}=\mathfrak{p}$. Therefore, by Theorem \ref{T3.5}, $W(\phi)$ is an irreducible smooth $\mathfrak{a}$-module of height $2$. Combining the conclusion of \cite[Theorem 4.7]{LCLLZZ} again, it can be deduced that $V$ is an irreducible smooth $\mathfrak{g}$-module of height $2$.
\end{proof}

\section* {Acknowledgement} 

W. Gao is partially supported by  Natural Science Foundation of Henan (No. 232300420351), Natural Science Foundation of China (No. 12471024) and Nanhu Scholars Program of XYNU (No. 012112), S. Liu is partially supported by Nanhu Scholars Program of XYNU (No. 012111), K. Zhao is  partially supported
by   NSERC (311907-2020), Y. Zhao is partially supported by National Natural Science Foundation of China (Grant No.12301040), and Nanhu Scholars Program of XYNU (No. 2021030).


	\
	
	Cunguang Cheng, School of Mathematical Sciences, Hebei Normal University, Shijiazhuang 050016, P. R. China. Email address: chengcg2024@163.com 
	
	\vskip 5pt
	
	Wenting Gao, School of Mathematics and Statistics,
	Xinyang Normal University, Xinyang 464000, P. R. China. Email address: gaowentingxy@163.com
	
	\vskip 5pt
	
	Shiyuan Liu, School of Mathematics and Statistics,
	Xinyang Normal University, Xinyang 464000, P. R. China. Email address: liushiyuanxy@163.com
	
	\vskip 5pt
	
	Kaiming Zhao,   Department of Mathematics, Wilfrid Laurier University, Waterloo, ON, Canada N2L3C5. Email address: kzhao@wlu.ca
	
	\vskip 5pt

	Yueqiang Zhao, School of Mathematics and Statistics,
	Xinyang Normal University, Xinyang 464000, P. R. China. Email address: yueqiangzhao@163.com
	
\end{document}